\documentclass{amsart}

\usepackage{amscd,amsmath,amssymb,amsthm,mathrsfs,latexsym}
\usepackage[pdftex]{graphicx}
\usepackage[mathcal]{euscript}
\usepackage[all]{xy}
\usepackage{color}
\usepackage{fancyhdr}
\usepackage{color}
\usepackage{graphicx}
\usepackage{caption}
\usepackage{subcaption}
\usepackage{color,soul} 
\usepackage{hyperref} 
\hypersetup{
 colorlinks=true, 	 
 linktoc=all,     	 
 linkcolor=black,  	 
 citecolor=black,    
 filecolor=magenta,  
 urlcolor=cyan       
}
\usepackage{tikz} 		
\usepackage{tikz-cd}
\usetikzlibrary{matrix,arrows,decorations.pathmorphing,intersections}

\usepackage{pifont} 
\usepackage{fourier-orns} 

\newtheorem{thm}{Theorem}[section]

\newtheorem{lemma}[thm]{Lemma}
\newtheorem{prop}[thm]{Proposition}
\newtheorem{cor}[thm]{Corollary}

\theoremstyle{definition}
\newtheorem{defn}[thm]{Definition}

\newtheorem{ex}[thm]{Example}
\newtheorem*{rmk}{Remark}

\newcommand{\Z}{{\mathbb{Z}}}
\newcommand{\R}{{\mathbb{R}}}
\newcommand{\C}{{\mathbb{C}}}

\newcommand{\Ho}{{\rm Hom}}

\newcommand{\colim}{{{\rm colim}\hspace{.2em}}}
\newcommand{\hocolim}{{{\rm hocolim}\hspace{.2em}}}

\newcommand{\Flag}{{{\rm Flag}}}

\newcommand*{\longhookrightarrow}{\ensuremath{\lhook\joinrel\relbar\joinrel\relbar\joinrel\rightarrow}}
\newcommand*{\longontoright}{\ensuremath{\relbar\joinrel\relbar\joinrel\twoheadrightarrow}}


\def\ds{\displaystyle}
\numberwithin{equation}{section} 

\title{On the fundamental group of certain polyhedral products}

\author{Mentor Stafa}
\address{Department of Mathematics, Tulane University, New Orleans, LA 70118}
\email{mstafa@tulane.edu}

\subjclass[2010]{Primary 55A05, 55U10, 55Q05.}

\keywords{polyhedral product, Eilenberg-Mac Lane space, transitively commutative group}

\thanks{The author was supported by DARPA grant number N66001-11-1-4132}
\thanks{This paper has been accepted for publication in the Journal of Pure and Applied Algebra.}

\begin{document}

\begin{abstract}
Let $K$ be a finite simplicial complex, and $(X,A)$ be a pair of spaces. The purpose of this article is to study the fundamental group of the polyhedral product denoted $Z_K(X,A)$, which denotes the moment-angle complex of Buchstaber-Panov in the case $(X,A) = (D^2, S^1)$, with extension to arbitrary pairs in \cite{cohen.macs} as given in Definition \ref{mac.defn} here.

For the case of a discrete group $G$, we give necessary and sufficient conditions on the abstract simplicial complex $K$ such that the polyhedral product denoted by $Z_K(\underline{BG})$ is an Eilenberg-Mac Lane space. The fundamental group of $Z_K(\underline{BG})$ is shown to depend only on the 1-skeleton of $K$. Further special examples of polyhedral products are also investigated.

Finally, we use polyhedral products to study an extension problem related to transitively commutative groups, which are given in Definition \ref{defn: TC group}.
\end{abstract}

\maketitle

\tableofcontents

\section{Introduction}

\

Let $(X,A)$ be a $CW$-pair with non-degenerate basepoint, and let $K$ be an abstract simplicial complex with a finite number of vertices. The \textit{polyhedral product functor} is a covariant functor defined on the category of $CW$-pairs with values in the category of topological spaces. That is, for any pair  $(X,A)$ and abstract finite simplicial complex $K$, the polyhedral product $Z_K(X,A)$ is a topological space with the natural choice of basepoint. This functor has been studied extensively and stands at the foundations of \textit{toric topology}. See for example work of V. Buchstaber and T. Panov \cite{buch.panov,buch.panov.toric.topology}, work of A. Bahri, M. Bendersky, F. Cohen and S. Gitler \cite{cohen.macs}, and elegant work of G. Denham and A. Suciu \cite{denham}.

\

One of the fundamental examples is the Davis-Januszkiewicz space 
$$\mathcal{DJ}(K) \simeq ET^m \times_{T^m} Z_K(D^2,S^1) \simeq Z_K(\C P^{\infty},\ast).$$
The cohomology ring of the space $Z_K(\C P^{\infty},\ast)$ is isomorphic as an algebra to the face ring of $K$, also known as the Stanley-Reisner ring. The equivalence between $\mathcal{DJ}(K)$ and $Z_K(\C P^{\infty},\ast)$ was proved by V. Buchstaber and T. Panov \cite{buch.panov}. The polyhedral product $Z_K(D^2,S^1)$ is called \textit{the moment-angle complex}, and is sometimes abbreviated $Z_K$ in the literature.

\

Seminal work of A. Bahri, M. Bendersky, F. Cohen and S. Gitler \cite{cohen.macs} set the solid framework necessary to study polyhedral products in broad generality. In particular, they study the stable decomposition of these spaces for any sequence of $CW$-pairs $(\underline{X},\underline{A})=\{(X_i,A_i)\}_{i=1}^n $. They prove a stable decomposition which has implications for computations in any cohomology theory. Namely, they prove that for a sequence $(\underline{X},\underline{A})$ of connected and pointed $CW$-pairs there is a pointed homotopy equivalence 
$$\Sigma( Z_K (\underline{X},\underline{A})) \to \Sigma ( \bigvee_{I \subseteq \{1,\dots,n\}}\widehat{Z}_{K_I} (\underline{X_I},\underline{A_I})),$$
where the ``hat'' stands for smash product (see Definition \ref{defn: smash polyhedral product}) and 
$$K_I=\{\sigma \cap I | \sigma \in K\}$$ 
is the restriction of $K$ to $I$ and $(\underline{X_I},\underline{A_I})$ is the restriction of the sequence of pairs to $I$, see \cite[Theorem 2.10]{cohen.macs}.

\

In this paper we study the homotopy groups of the polyhedral products $Z_K(\underline{X},\underline{A})$ with a focus on the fundamental group for specific choices of $CW$-pairs $(\underline{X},\underline{A})$.

\

Let $G$ be a group with the discrete topology and let $BG$  denote its classifying space with $\ast=B\{1\}$ as the basepoint. If $G$ is not discrete, then $BG$ might have non-trivial higher homotopy groups. Let $EG$ be a contractible topological space on which $G$ acts freely and properly discontinuously. If $G_1,\dots,G_n$ are discrete groups, then we show that the existence of the higher homotopy groups of the polyhedral product $Z_K(\underline{BG})$ depends only on $K$. In particular, if $K$ is a \textit{flag complex} (see Definition \ref{defn: flag cx}), then higher homotopy groups of $Z_K(\underline{BG})$ vanish, see \cite[Example 1.6]{davis.okun}. Here we show that the condition that $K$ is a flag complex is necessary.

\begin{thm}\label{thm: Z_K is an E-M space iff K flag INTRO}
Let $G_1,\dots , G_n$ be non-trivial discrete groups and $K$ be a simplicial complex with $n$ vertices. Then  $Z_K({\underline{BG}})$ is an Eilenberg-Mac Lane space if and only if $K$ is a flag complex. Equivalently, $Z_K(\underline{EG},\underline{G})$ is an Eilenberg-Mac Lane space if and only if $K$ is a flag complex.
\end{thm}

Let $SK_1$ denote the 1-skeleton of the simplicial complex $K$. In \cite{davis.okun} it was shown that if $K$ is a flag complex, then the fundamental group of the polyhedral product $Z_K(\underline{BG})$ is the \textit{graph product} of the groups $G_1,\dots,G_n$ over the graph $SK_1$ (see Definition \ref{defn: graph product of gps}), that is,
$$\pi_1(Z_K(\underline{BG})) \cong \prod_{SK_1}G_i.$$
Here we show that this is the case for any abstract simplicial complex $K$ on a finite number of vertices. That is, the fundamental group in this case is determined by the 1-skeleton of $K$. The higher homotopy groups may appear only when $K$ contains $m$-dimensional faces for $m \geq 2$.

\begin{prop}\label{prop: fun.gp.Zk INTRO}
Let $K$ be a simplicial complex on $n$ vertices and let $SK_1$ be the 1-skeleton of $K$. Let $G_1,\dots,G_n$ be discrete groups. Then
\[\pi_1 (Z_K(\underline{BG})) \cong \pi_1 (Z_{SK_1}(\underline{BG})) \cong \prod_{SK_1} G_i.\]
\end{prop}

The fundamental group of the polyhedral products $Z_K({X})$ is also investigated for a 1-connected space $X$. Let $CY$ denote the cone on the topological space $Y$. Then the following theorem holds.

\begin{thm}\label{thm: pp is 1-connected INTRO}
Let $X_1,\dots,X_n$ be 1-connected $CW$-complexes. Then the polyhedral product 
$Z_K(\underline{C \Omega X},\underline{\Omega X})$ is 1-connected.
\end{thm}

Since the polyhedral product $Z_K(\underline{C \Omega X},\underline{\Omega X})$ is 1-connected, one can describe the first non-vanishing homotopy group using the Hurewicz isomorphism. For this choice of pairs of spaces the first non-vanishing homotopy group depends on the structure of the simplicial complex $K$ as follows.

Let $I=\{i_1,\dots,i_k\}\subset \{1,\dots,n\}$ and for a sequence of spaces given by $\underline{X}=\{X_1,\dots,X_n\}$ let $\underline{X}^I=X_{i_1} \wedge \cdots \wedge X_{i_k}$. Then the first non-vanishing homotopy group of the polyhedral product $Z_K(\underline{C \Omega X},\underline{\Omega X})$ is given by the following proposition.

\begin{prop}\label{prop: 2nd homotopy group, special case INTRO}
Let $X_1,\dots,X_n$ be 1-connected $CW$-complexes and $K$ be a simplicial complex with $n$ vertices. If $Z_K(\underline{C \Omega X},\underline{\Omega X})$ is $(k-1)$-connected, then the first non-vanishing homotopy group is given by 
$$\pi_k( Z_K(\underline{C \Omega X},\underline{\Omega X})) \cong \bigoplus_{I \subset \{1,\dots,n\}}H_k(|K_I| \wedge \widehat{\underline{\Omega X}}^I; \Z).$$
\end{prop}

Now assume that the fundamental group of $X$ and $Y$ for the pairs $(X,A)$ and $(Y,A)$ are isomorphic and there are inclusions $A\subset X \subset Y$. Using the colimit definition of polyhedral products and the fact that homotopy colimit and colimit commute in this case, one can restrict to a subspace if necessary, to compute the fundamental group of the polyhedral product.

\begin{lemma}\label{lemma: f. group of subspace of poly. prod. INTRO}
If $A \subset X \subset Y $ with $A$ discrete, $X$ and $Y$ path-connected, such that the inclusion induces an isomorphism  of the fundamental groups of $X$ and $Y$, then there is an isomorphism
\[\pi_1(Z_K(X,A)) \cong \pi_1(Z_K(Y,A)).\]
\end{lemma}

Suppose that $H$ is a closed subgroup of the Lie group $G$ and consider the pair $(BG,BH)$. There is an inclusion $G/H \hookrightarrow C(G/H)$ which gives an inclusion 
$$EG\times G/H \hookrightarrow EG \times C (G/H).$$
So we can regard $G/H$ as a $G$-equivariant subspace of $EG$. The following is an analogue of the Denham-Suciu fibration.

\begin{prop}\label{prop: analogue of Denham-Suciu INTRO}
Let $H$ be a closed subgroup of the Lie group $G$. There is a fibration given by
\[ Z_K(EG,G/H) \longrightarrow Z_K(BG,BH)  \longrightarrow (BG)^n .\]
\end{prop}

As a consequence of Propositionn \ref{prop: analogue of Denham-Suciu INTRO}, we obtain the following splitting and short exact sequence of groups.

\begin{prop}\label{prop: analog of DS, splitting & ses INTRO}
Let $H$ be a closed subgroup of $G$. Then there is a homotopy equivalence
\[\Omega ( Z_K(BG,BH) ) \simeq G^n \times \Omega (Z_K(EG,G/H))\] 
and a short exact sequence of groups
\[1 \to\pi_1(Z_K(EG,G/H)) \to  \pi_1( Z_K(BG,BH) ) \to G^n \to 1.\]
Furthermore, this decomposition may not preserve the group structure.
\end{prop}

\

A group $G$ is called \textit{transitively commutative} if for any $g,h,k \in G$, $[g,h]=1$ and $[h,k]=1$ imply $[g,k]=1$. For $G$ a finite discrete groups and $G_1,\dots, G_n$ a collection of subgroups of $G$, there is a natural map 
$$ 
\varphi: G_1 \ast\dots \ast G_n \to G
$$
which sends words to their product in $G$.
There is also a natural map
$$ 
p: \prod_{\Gamma}G_i \to G,
$$
from the \textit{graph product} of $G_1,\dots,G_n$ (see Definition \ref{defn: graph product of gps}) to $G$, where $\Gamma$ is a simplicial graph, i.e. a 1-dimensional simplicial complex, with $n$ vertices. Similarly, the map $p$ sends words representing cosets to their product in $G$.
We investigate the problem when the map $\varphi$ can factor through $p$, that is through the graph product of groups $\prod_{\Gamma} G_i$. This is equivalent to finding the graphs $\Gamma$ such that the following diagram commutes
\begin{center}
\begin{tikzcd}
G_1 \ast \cdots \ast G_n \arrow{r}{\varphi} \arrow{d}[swap]{p} &G. \\
\prod_{\Gamma}G_i \arrow[dotted]{ur}
\end{tikzcd}
\end{center}
Among other things, answering this question gives information about commuting elements in $G$. We study this problem for the class of transitively commutative groups. More precisely, the aim is to give topological conditions characterizing transitively commutative groups, using polyhedral products. 

Moreover, we define an extension of the group theoretic concept of \textit{transitively commutative} groups, to \textit{k-transitively commutative} groups, which is a group theoretic construction reflected by the natural topology of the spaces in this article. Features of \textit{k-transitively commutative} groups are investigated after they are introduced, in Section \ref{section.extension}.

\

The organization of this article is as follows. In Section \ref{sec: polyhedral products} polyhedral products are introduced and some of their properties are given. In Section \ref{sec: htpy groups Z_K for BG} Theorem \ref{thm: Z_K is an E-M space iff K flag INTRO} and Proposition \ref{prop: fun.gp.Zk INTRO} are proved. In Section \ref{sec: htpy groups other Z_K} we prove Theorem \ref{thm: pp is 1-connected INTRO}, Proposition \ref{prop: 2nd homotopy group, special case INTRO}, Lemma \ref{lemma: f. group of subspace of poly. prod. INTRO} and Propositions \ref{prop: analogue of Denham-Suciu INTRO} and \ref{prop: analog of DS, splitting & ses INTRO}. Finally, in Section \ref{section.extension} we study the extension problem for transitively commutative groups.

\section{Polyhedral products}\label{sec: polyhedral products}

\


Polyhedral products can be \textit{loosely} regarded as functors from the category of topological spaces to the category of topological spaces, if the value of $K$ is fixed.
Here the discussion is confined to pointed $CW$-pairs $(X,A)$, where  
$(\underline{X},\underline{A} )$ denotes a finite sequence of pointed $CW$-pairs $\{(X_i,A_i)\}_{i=1}^n$. By $[n]$ we denote the set $\{1,2,\dots,n\}$, and by $2^{[n]}$ we mean the power set of $[n]$.

\begin{defn}
An \textit{abstract simplicial complex} (or only \textit{simplicial complex}) $K$ on $n$ vertices is a subset of the power set $2^{[n]}$ such that, if $\sigma \in K$ and $\tau \subseteq \sigma$ then $\tau \in K$. An element $\sigma \in K$, called a \textit{simplex}, is given by an increasing sequence of integers $\sigma=\{i_1, i_2, \dots ,i_q\}$, where $1 \leq i_1 < i_2 < \cdots < i_q \leq n$. In particular, the empty set $\varnothing$ is an element of $K$. The \textit{geometric realization} $|K|$ of $K$ is a simplicial subcomplex of the simplex $\Delta[n-1]$.
\end{defn}

Define a functor $D$ from an abstract simplicial complex $K$ to the category of pointed $CW$-complexes, $D: K \longrightarrow CW_{\ast}$ as follows. For any $\sigma \in K$ let 

\begin{tabular}{c c c}
$\ds D(\sigma)=\prod_{i=1}^n Y_i=Y_1 \times \cdots \times Y_n \subseteq \prod_{i=1}^n X_i,$
& where &
 $
  Y_i = \left\{\def\arraystretch{1.2}%
  \begin{array}{@{}c@{\quad}l@{}}
    A_i & : i \notin \sigma ,\\
       X_i & : i \in \sigma .\\
  \end{array}\right.
$
\end{tabular}

\begin{defn}\label{mac.defn}
The \textit{polyhedral product} (or \textit{generalized moment-angle complex} \cite{denham}) denoted $Z_K(\underline{X},\underline{A})$ is the topological space defined by the colimit
\[ Z_K(\underline{X},\underline{A})= \underset{{\sigma \in K}}{\colim}D(\sigma) = \bigcup_{\sigma \in K} D(\sigma)  \subseteq \prod_{i=1}^n X_i, \]
where the maps are the inclusions and the topology is the subspace topology.
\end{defn}

Similarly define a functor $\widehat{D}: K \longrightarrow CW_{\ast}$, so that for any $\sigma \in K$ we have 

\begin{tabular}{c c c}
$\ds \widehat{D}(\sigma)=\prod_{i=1}^n Y_i=Y_1 \wedge \cdots \wedge Y_n \subseteq \prod_{i=1}^n X_i,$
& where &
 $
  Y_i = \left\{\def\arraystretch{1.2}%
  \begin{array}{@{}c@{\quad}l@{}}
    A_i & : i \notin \sigma ,\\
       X_i & : i \in \sigma .\\
  \end{array}\right.
$
\end{tabular}

\

\begin{defn}\label{defn: smash polyhedral product}
The \textit{smash polyhedral product} $\widehat{Z}_K(\underline{X},\underline{A})$ is defined to be the image of the polyhedral product ${Z}_K(\underline{X},\underline{A})$ in the smash product $X_1 \wedge \cdots \wedge X_n$. For example, the image of $D(\sigma)$ in $\widehat{Z}_K(\underline{X},\underline{A})$ is given by $\widehat{D}(\sigma)$.
\end{defn}

In other words the polyhedral product is the colimit of the diagram of spaces $D(\sigma)$. Different notations can be found in the literature to denote a polyhedral product, such as $Z_K(X_i,A_i)$, $Z(K;(\underline{X},\underline{A}))$ and $(\underline{X},\underline{A})^K$. Whenever $A_i$ is the basepoint of $X_i$, we will write $Z_K(\underline{X})$ to mean $Z_K(\underline{X},\underline{\ast})$, or we simply write $Z_K({X})$ if all the spaces 
{\bf $X_i$} in the sequence of pairs are the same.

\begin{ex} Some classical examples of polyhedral products are the following.
\begin{enumerate}
\item Let $K=\{ \{1\},\dots,\{n\} \}$, $X_i=X$ and $A_i$ be the basepoint $\ast \in X$. Then 
$$Z_K(\underline{X},\underline{A})=X \vee \dots \vee X,$$
the $n$-fold wedge sum of the space $X$. 
\item Let $K=2^{[n]}$, then $Z_K(\underline{X},\underline{A})=X_1 \times \cdots \times X_n$.
\item Let $K=\{ \{1\},\{2\} \}$ and $(\underline{X},\underline{A})=(D^n,S^{n-1})$. Then 
\[Z_K(\underline{X},\underline{A})=D^n \times S^{n-1} \cup S^{n-1}\times D^n=\partial D^{2n}=S^{2n-1}.\]
\item If $|K|$ is the boundary of the simplex $\Delta[n-1]$, then $Z_K(\underline{X})$ is the fat wedge of the spaces $X_1,\dots , X_n$.
\end{enumerate}
\end{ex}

Polyhedral products give combinatorial models for some well-known spaces. As mentioned in the introduction, the Davis-Januszkiewicz space $\mathcal{DJ}(K)$ is one such space since it is homotopy equivalent to $Z_K(\mathbb{C}P^{\infty},\ast)$.

\begin{defn} \label{defn: flag cx}
The simplicial complex $|K|$ is a \textit{flag complex} if any finite set of vertices, which are pairwise connected by edges, spans a simplex in $|K|$.
\end{defn}

\begin{defn}\label{defn: graph product of gps}
Given a simplicial graph $\Gamma$ with vertex set $S$ and a family of groups $\{G_s\}_{s\in S}$, their \textit{graph product} 
\[\ds \prod_{\Gamma} G_s \]
is the quotient of the free product of the groups $G_s$ by the relations that elements of $G_s$ and $G_t$ commute whenever $\{s, t\}$ is an edge of $\Gamma$.
Note that if $\Gamma$ is the complete graph on $n$ vertices, then 
\[\ds \prod_{\Gamma} G_s \cong \prod_{i=1}^n G_i.\]
\end{defn}

Another example of interest is the classifying space of the right-angled Artin group (RAAG). Recall that an RAAG can be given in terms of generators and relations or as a graph product of a finite number of copies of the infinite cyclic group $\Z$.  The fundamental group of the polyhedral product $Z_K(S^1,\ast)$ is the right-angled Artin group given by the graph product $\prod_{SK_1 } \mathbb{Z}$, where $SK_1$ is the 1-skeleton of $K$. If $K$ is a flag complex, then $Z_K(S^1,\ast)$ is the classifying space of the right-angled Artin group, see \cite{davis.okun}. For other finite classifying spaces of Artin groups see \cite{charney.davis}.

\

In our discussion, $K$ will frequently mean the geometric realization of $K$. In what follows, assume that spaces are in the homotopy category of pointed $CW$-complexes unless otherwise stated.

\

\section{Pairs of classifying spaces of discrete groups}\label{sec: htpy groups Z_K for BG}

\

Let $G$ be a discrete group. In this section we consider polyhedral products constructed from pairs of spaces of the form $(BG,\ast)$, where $\ast$ is the natural basepoint, and study their homotopy groups. 
More precisely, we are interested in describing the fundamental group of $Z_K(\underline{BG})$, where $(\underline{BG},\underline{\ast})=\{ (BG_i,\ast_i)\}_{i=1}^n$ and $G_1,\dots,G_n$ are discrete groups.  Later we find necessary and sufficient conditions for the simplicial complex $K$ such that $Z_K(\underline{BG})$ is an Eilenberg-Mac Lane space. 

\

In general, for $G$ a topological group and $K$ a simplicial complex  on $n$ vertices, there is a fibration given by G. Denham and A. Suciu in \cite[Lemma 2.3.2]{denham}.
\begin{thm}\label{fibration.theorem}
Let $G$ be a topological group and $K$ a simplicial complex  on $n$ vertices. Then the following hold.
\begin{enumerate}
\item $EG^{ n} \times_{G^{ n}} Z_K(EG,G) \simeq Z_K(BG)$.
\item The homotopy fiber of the inclusion $Z_K(BG) \hookrightarrow BG^{n}$ is $Z_K(EG,G)$.
\end{enumerate}
\end{thm}
Their theorem is a result of studying the bundle 
\[Z_K(EG,G)\longrightarrow  EG^{n} \times_{G^{ n}} Z_K(EG,G) \longrightarrow BG^n,\]
where $EG^{n} \times_{G^{ n}} Z_K(EG,G)$ is the quotient of $ EG^{n} \times Z_K(EG,G)$ by the diagonal action of $G^n$, and proving that the total space $EG^{n} \times_{G^{ n}} Z_K(EG,G)$ is homotopy equivalent to $Z_K(BG)$. This theorem can be extended directly to the sequence of $CW$-pairs 
$(\underline{BG})=\{({BG_i},\ast_i)\}_{i=1}^n$, where $G_i$ are topological groups for all $i$. Hence, there is a fibration
\begin{equation}\label{D-S-fibration}
Z_K(\underline{EG},\underline{G}) \longrightarrow Z_K(\underline{BG}) \longrightarrow {BG_1}\times \cdots \times BG_n,
\end{equation}
where similarly, the total space is homotopy equivalent to the twisted product of spaces
\[(EG_1 \times \cdots \times EG_n) \times_{(G_1 \times \cdots \times G_n)} Z_K(\underline{EG},\underline{G}).\]
We will refer to fibration (\ref{D-S-fibration}) as the \textit{Denham-Suciu fibration}, even though this fibration stems from earlier ideas in the study of moment-angle complexes by V. Buchstaber and T. Panov \cite{buchstaber2002torus}, as well as work of M. Davis and T. Januszkiewicz \cite{davis.januszckiewicz}.

\

\begin{lemma}\label{htpy.gps.Zk} Assume that $G_1,\dots,G_n$ are discrete groups and $K$ is a simplicial complex on $n$ vertices. The following properties hold:
\begin{enumerate}
\item $Z_K(\underline{BG})$ is path connected for any $K$.
\item $\pi_q(Z_K(\underline{EG},\underline{G})) \cong \pi_q(Z_K(\underline{BG}))$ for $q\geq 2$.
\item There is a short exact sequence of groups 
\[1 \rightarrow  \pi_1 (Z_K(\underline{EG},\underline{G})) \rightarrow \pi_1 (Z_K(\underline{BG})) \rightarrow \pi_1 (BG_1\times \cdots \times BG_n) \rightarrow 1. \]
\end{enumerate}
\end{lemma}
\begin{proof}
This follows directly from the long exact sequence in homotopy applied to the Denham-Suciu  fibration (\ref{D-S-fibration}) and the fact that there is an equivalence $BG_1\times \cdots \times BG_n=B(G_1\times \cdots \times G_n)$.
\end{proof}

Consider the following definition, which will give a different characterization of flag complexes.

\begin{defn}
Let $K$ be a simplicial complex on $n$ vertices. A \textit{minimal non-face} in $K$ is a subcomplex of $K$, which is the boundary of a simplex, but the simplex itself is not in $K$. For the purposes in this paper all minimal non-faces will be assumed to have at least three vertices.
\end{defn}

\begin{rmk}
Clearly if $K$ has a minimal non-face with three or more vertices, then it cannot be a flag complex. Therefore, a flag complex can be redefined to be a simplicial complex with no minimal non-faces.
\end{rmk}

The fundamental group of the spaces $Z_K(\underline{BG})$ is calculated next.

\begin{prop}\label{fun.gp.Zk}
Let $K$ be a simplicial complex on $n$ vertices and let $SK_1$ be the 1-skeleton of $K$. If $G_1,\dots,G_n$ are discrete groups, then
\[\pi_1 (Z_K(\underline{BG})) \cong \pi_1 (Z_{SK_1}(\underline{BG})) \cong \prod_{SK_1} G_i.\]
\end{prop}
\begin{proof}
Consider the following commutative diagram of spaces
\begin{center}
\begin{tikzcd}
F \arrow{r} \arrow{d}  & Z_{SK_q} (\underline{EG},\underline{G}) \arrow{r} \arrow{d}  & Z_K(\underline{EG},\underline{G}) \arrow{d} \\
F \arrow{r} \arrow{d} & Z_{SK_q} (\underline{BG}) \arrow{r} \arrow{d} & Z_K(\underline{BG}) \arrow{d} \\
\ast \arrow{r} & \prod_{i=1}^n {BG_i} \arrow{r} & \prod_{i=1}^n {BG_i} ,
\end{tikzcd}
\end{center}
where $F$ is the homotopy fibre and $SK_q$ is the $q$-skeleton of $K$, for $1 \leq q \leq n$. Note that the homotopy fibre $F$ of the inclusion 
\[ Z_{SK_q}(\underline{EG},\underline{G}) \longhookrightarrow Z_K(\underline{EG},\underline{G})\]
is homotopy equivalent to the homotopy fibre $F$ of the inclusion
\[ Z_{SK_q}(\underline{BG}) \longhookrightarrow Z_K(\underline{BG}).\]
In particular, for the 1-skeleton $SK_1$ there is an inclusion
$$
Z_{SK_1}(\underline{BG}) \longhookrightarrow Z_K(\underline{BG})).
$$
It follows from the Seifert-Van Kampen theorem, applied finitely many times to the polyhedral product, that the fundamental group $\pi_1(Z_{SK_1}(\underline{BG}))$ is isomorphic to the graph product 
$
 \prod_{SK_1} G_i.
$

The following diagram obtained from above
\begin{center}
\begin{tikzcd}
Z_{SK_q} (\underline{EG},\underline{G}) \arrow{r} \arrow{d}  & Z_K(\underline{EG},\underline{G}) \arrow{d} \\
Z_{SK_q} (\underline{BG}) \arrow{r} & Z_K(\underline{BG})
\end{tikzcd}
\end{center}
commutes. It suffices to show that on the level of fundamental groups, the induced map
$\pi_1(Z_{SK_1}(\underline{EG},\underline{G})) \to \pi_1(Z_K(\underline{EG},\underline{G}))$
is an injection. It will then follow that the map
$
\pi_1(Z_{SK_1}(\underline{BG})) \to \pi_1(Z_K(\underline{BG})))
$
is both an injection and surjection since $F$ is path-connected.

Note that $Z_{SK_1}(\underline{EG},\underline{G})$ is a 2-dimensional $CW$-complex and adding 2-dimensional or higher dimensional faces $\sigma$ to the $1$-skeleton $SK_1$ only adds 3-dimensional or higher dimensional cells $D(\sigma)$ to the space $Z_{SK_q}(\underline{EG},\underline{G})$ (see Definition \ref{mac.defn}). Therefore, adding these cells to the polyhedral product  does not change the fundamental group, see \cite[e.g. Corollary 4.12 and Example 4.16]{hatcher2002algebraic}.
\end{proof}

\subsection{{Zero simplicial complexes}}

For a fixed abstract simplicial complex $K$, the polyhedral product is a homotopy functor, which means that for fixed $K$ the homotopy type of $Z_K(\underline{X},\underline{A})$ depends only on the relative homotopy type of the pairs $(\underline{X},\underline{A})$. This fact was also observed in \cite[Section 2.2]{denham}. In particular, this is true for the $0$-simplicial complex on $n$ vertices and for the pair $(EG,G)$, where $G$ is a topological group. The following lemma proves a special case of this fact, which will be used in a later discussion.

\begin{lemma}\label{lemma: rel homotopy equiv}
Let $G$ be a finite discrete group of order $m$. Then there is a relative homotopy equivalence $(EG,G)\sim ([0,1],F)$, where $F$ is a discrete subset of $[0,1]$ with $m$ elements.
\end{lemma}
\begin{proof}
By definition $EG$ is contractible and the group $G$ can be identified with the orbit of a point $x \in EG$, since $G$ acts freely on $EG$. This gives an equivalence of $CW$-pairs $(EG,G) \simeq ([0,1],F)=(I,F)$, where $$F=\{f_1<\dots< f_m\} \approx \{(1=g_1)\cdot x=x,g_2\cdot x,\dots,g_m\cdot x\}$$ is a finite subset of $I$ with the same cardinality as $G$. One can pick a path $\gamma: [0,1]\to EG$ such that $\gamma(f_i)=g_i \cdot x$. In particular, $F$ can be chosen to be $F=\{0=f_1<\dots<f_i<\cdots< f_m=1\}$.
\end{proof} 

Hence, if $K$ has $n$ vertices, there is an equivalence $Z_K(\underline{EG},\underline{G})\simeq Z_K(\underline{I},\underline{F})$, where $(\underline{I},\underline{F})=\{(I,F_i)\}_{i=1}^n$ and $|F_i|=|G_i|$.

\begin{prop}\label{prop: poly prod for 0-skeleton} Let $K$ be a discrete set of $r$ points and $F_i$ be finite subsets of $I=[0,1]$ with $|F_i|=m_i$, $F_i=\{0=f_{1}^i<\dots<f_{j}^i<\cdots f_{m_i}^i=1\}$, for $1 \leq i \leq r$. Then there is a homotopy equivalence
$$Z_K (\underline{I},\underline{F}) \simeq \bigvee_{N_r}S^1,$$ 
where $N_r$ is given inductively as follows:
\begin{align*}
N_2&=(m_1-1)(m_2-1), \\
N_r&=m_rN_{r-1}+(m_r-1)(\prod_{i=1}^{r-1} {m_i} -1), \text{ for } r \geq 3. 
\end{align*}

\end{prop}
\begin{proof} Recall that if $T$ is a spanning tree of a connected graph $\Gamma$ with a finite number of vertices, then collapsing $T$ to a point does not change the homotopy type of $\Gamma$. $\Gamma / T$ has only one vertex and has the homotopy type of a finite wedge of circles. Since $T$ is contractible, then $\Gamma \simeq \Gamma/T$.

To complete the proof it suffices to find the number of circles $N_r$ and the proof is given by induction on r. 

\begin{figure}[htbp]
\centering
\begin{minipage}{0.4\linewidth}
\centering
{
\begin{tikzpicture}
\draw[dotted](0,0)--(0,1);
\draw[dotted](0,2)--(0,3);
\draw[dotted](3,1)--(3,2);
\draw[dotted](1,0)--(1,3);
\draw[dotted](2,0)--(2,3);
\path[draw,thick](0,0)--(3,0);
\path[draw,thick](3,0)--(3,1);
\path[draw,thick](0,1)--(3,1);
\path[draw,thick](0,1)--(0,2);
\path[draw,thick](3,2)--(0,2);
\path[draw,thick](3,2)--(3,3);
\draw[->,thick](3,3)--(0,3);
\coordinate [label=below left:$0$] (1) at (0,0);
\coordinate [label=left:$f_{2}^2$] (2) at (0,1);
\coordinate [label=left:$\vdots$] (3) at (0,2);
\coordinate [label=left:$f_{m_2}^2$] (5) at (0,3);
\coordinate [label=below:$f_{2}^1$] (2) at (1,0);
\coordinate [label=below:$\dots$] (3) at (2,0);
\coordinate [label=below:$f_{m_1}^1$] (5) at (3,0);
\end{tikzpicture}
}
\caption{$T_2,r=2$}\label{fig:t2}
\end{minipage}%
\begin{minipage}{0.4\linewidth}
\centering
{
\vspace*{1em}
\begin{tikzpicture}[scale=.94]
\path[draw,thick](0,0)--(3,0);
\path[draw,thick](3,0)--(5,1);
\path[draw,thick](5,1)--(2,1);
\path[draw,thick](2,1)--(2,3);
\path[draw,thick](2,3)--(5,3);
\path[draw,thick](5,3)--(3,2);
\draw[->,thick](3,2)--(0,2);
\draw[dotted](0,0)--(0,2);
\draw[dotted](0,0)--(2,1);
\draw[dotted](3,0)--(3,2);
\draw[dotted](5,1)--(5,3);
\draw[dotted](0,2)--(2,3);
\coordinate [label=below left:$0$] (1) at (0,0);
\coordinate [label=below:$\hdots$] (3) at (1.5,0);
\coordinate [label=below:$f_{m_1}^1$] (5) at (3,0);
\coordinate [label=left:$f_{m_2}^2$] (5) at (2,1);
\coordinate [label=left:$f_{m_3}^3$] (5) at (0,2);
\coordinate [label=left:$\vdots$] (5) at (0,1);
\coordinate [label=left:$.$] (1) at (1,0.5);
\coordinate [label=left:$.$] (2) at (0.86,0.43);
\coordinate [label=left:$.$] (3) at (1.14,0.57);\end{tikzpicture}
}
\vspace*{-1em}

\caption{$T_3,r=3$}\label{fig:t3}
\end{minipage}
\end{figure}

$r=2$: $K$ has two vertices and $Z_K(\underline{I},\underline{F})$ contains a maximal tree, which we denote $T_2$, defined in the following way: It starts at the point $(0,0)\in I\times I$ and runs parallel to the first coordinate and goes to the next level (i.e planes $y=f_{2}^2,\dots , y=f_{m_2}^2$ ) by using one of the extreme vertical edges. $T_2$ contains all the vertices and has no loops, hence it is a spanning tree. There are $N_2=(m_1-1)(m_2-1)$ edges not in $T_2$. Figure \ref{fig:t2} shows an interpretation of $T_2$.

$r=3$: $K$ has three vertices and $Z_K(\underline{I},\underline{F})$ contains a maximal tree called $T_3$ (see Figure \ref{fig:t3}) defined in the similar way as above: on each level parallel to the $xy$-plane it is the same as $T_2$ and it needs a vertical edge to jump to the next dimension each time. There are $m_3$ levels, each having $N_2$ edges not in $T_3$, and there are $m_3-1$ spaces between levels, each having $\prod^{2}_{i=1}(m_i)-1$ edges not in $T_3$. Therefore there are $N_3=m_3N_2+(m_3-1)(\prod^{2}_{i=1}(m_i)-1)$. 

Assume true for $r=n$, that is, $N_n=m_{n}N_{n-1}+(m_n-1)(\prod^{n-1}_{i=1}(m_i)-1)$. Let $K_{[n]}$ denote a finite set with $n$ points. For $r=n+1$ there is an inclusion $Z_{K_{[n+1]}}(\underline{I},\underline{F})\subseteq \mathbb{R}^{n+1}$. Set
\[ A_n=Z_{K_{[n]}}(\underline{I},\underline{F}) \subseteq\mathbb{R}^{n},\] 
then $A_n\simeq \bigvee _{N_n}S^1$. Figure \ref{fig:A_n+1} describes $A_{n+1}$,
\begin{figure}[ht!]
\centering
\begin{tikzpicture}
\draw[line width=1pt] (0,0) rectangle (2,2);
\node (A) at (1,1) {$A_n$};
\path[draw](2,2)--(3.5,2);
\path[draw](2,1.8)--(3.5,1.8);
\path[draw](2,1.6)--(3.5,1.6);
\path[draw](2,.2)--(3.5,.2);
\path[draw](2,0)--(3.5,0);
\draw[dotted](2.75,1.4)--(2.75,0.4);
\draw[line width=1pt] (3.5,0) rectangle (5.5,2);
\node (A) at (4.5,1) {$A_n$};
\path[draw](5.5,2)--(7,2);
\path[draw](5.5,1.8)--(7,1.8);
\path[draw](5.5,1.6)--(7,1.6);
\path[draw](5.5,.2)--(7,.2);
\path[draw](5.5,0)--(7,0);
\draw[dotted](6.25,1.4)--(6.25,0.4);
\path[draw](8,2)--(9.5,2);
\path[draw](8,1.8)--(9.5,1.8);
\path[draw](8,1.6)--(9.5,1.6);
\path[draw](8,.2)--(9.5,.2);
\path[draw](8,0)--(9.5,0);
\draw[dotted](8.75,1.4)--(8.75,0.4);
\draw[dotted](7.2,.9)--(7.8,.9);
\draw[line width=1pt] (9.5,0) rectangle (11.5,2);
\node (A) at (10.5,1) {$A_n$};
\end{tikzpicture}
\caption{$A_{n+1}$} \label{fig:A_n+1}
\end{figure}
where between any two consecutive $A_n$'s there are $\prod^{n}_{i=1}m_i$ edges. Each $A_n$ contains the maximal tree $T_n$ (they all lie on different planes $y=f_{j}^{n+1}$) and the number of edges in $A_n$ not in $T_n$ is $N_n$. One edge is used between two consecutive $A_n$'s to complete the graph $T_{n+1}$, so there are $(\prod^{n}_{i=1}m_i)-1$ edges not in $T_{n+1}$. Hence the total number of edges not in $T_{n+1}$ is $N_{n+1}=m_{n+1}N_{n}+(m_{n+1}-1)(\prod^{n}_{i=1} m_i-1)$.
\end{proof} 
\begin{cor}\label{cor: fibre rank} The value of $N_r$ in Proposition \ref{prop: poly prod for 0-skeleton} is
\[N_r=(r-1)\prod_{i=1}^r m_i - \sum_{i=1}^r (\prod_{j\neq i} m_j)+1.\]
\end{cor}
\begin{proof} The proof follows by induction on $r$. For $r=2$ then $(m_1-1)(m_2-1)=m_1m_2 -(m_1+m_2)+1$. Now assume this is true for $r=n$, then for $r=n+1$ it follows from Lemma \ref{htpy.gps.Zk} that
\[N_{n+1}=m_{n+1}N_{n}+(m_{n+1}-1)(\prod_{i=1}^{n} {m_i} -1),\] 
where by assumption $N_n$ equals
\[N_n=(n-1)\prod_{i=1}^n m_i - \sum_{i=1}^n (\prod_{j\neq i} m_j)+1.\]
Substituting this value for $N_n$ and rearranging the terms shows that
\[N_{n+1}= n\prod_{i=1}^{n+1} m_i - \sum_{i=1}^{n+1} (\prod_{j\neq i} m_j)+1.\]
\end{proof}

\begin{thm}\label{thm: htpy type of Z_K_0(EG,G)}
Let $G_1,\dots,G_r$ be finite discrete groups with $|G_i|=m_i$, for all $i$, and $K$ be a discrete set of $r$ points. Then there is a homotopy equivalence
$$
Z_K(\underline{EG},\underline{G}) \simeq \bigvee_{N_r}S^1,
$$ 
where $N_r$ is given in Corollary \ref{cor: fibre rank}.
\end{thm}
\begin{proof}
Follows from Proposition \ref{prop: poly prod for 0-skeleton} and Corollary \ref{cor: fibre rank}
\end{proof}

Consider the Denham-Suciu fibration (\ref{D-S-fibration}) for $K$ the 0-skeleton, to get the following fibration
\[
Z_K(\underline{EG},\underline{G}) \simeq \bigvee_{N_n} S^1 \longrightarrow \bigvee_{1\leq i\leq n} {BG_i} \longrightarrow {BG_1}\times \cdots \times BG_n.
\]
Each of the spaces in the fibration is an Eilenberg-Mac Lane space, hence there is a short exact sequence of groups
\[
1 \longrightarrow F[x_1,\dots,x_{N_n}] \longrightarrow G_1 \ast \cdots \ast G_n \longrightarrow G_1\times \cdots \times G_n \longrightarrow 1.
\]
The rank of the free group in the kernel is $N_n$, which is given in Corollary \ref{cor: fibre rank}. This rank was also computed algebraically in an early paper of J. Nielsen \cite{nielsen}. Thus Proposition \ref{prop: poly prod for 0-skeleton} gives a topological proof of that same result.

\subsection{{Flag complexes}}

Suppose $K$ is a flag complex. If $G$ is a non-trivial discrete group then the space $Z_K(BG,\ast)$ is a $K(\pi,1)$, see \cite[Example 1.6]{davis.okun}. By a short argument in Lemma \ref{lemma: no flag no K(pi,1)}, the converse of this statement is also true. 

\begin{lemma}\label{lemma: no flag no K(pi,1)}
If $K$ is not a flag complex, then $Z_K(EG,G)$ is not an Eilenberg-Mac Lane space.
\end{lemma}
\begin{proof}
There is an inclusion of pairs $([0,1],\{0,1\})\hookrightarrow (EG,G)$. If $K$ is not a flag complex, then $K$ contains a minimal non-face $\sigma \subseteq K$ with $k$ vertices, where $k\geq 3$. Then there is an embedding 
$$
Z_{\sigma}([0,1],\{0,1\})\hookrightarrow Z_K(EG,G).
$$
There is an equivalence $Z_{\sigma}([0,1],\{0,1\})\simeq S^{k-1}$. Therefore, $Z_K(EG,G)$ has non-trivial higher homotopy groups.
\end{proof}

\begin{rmk}
If $K$ is a flag complex, then $Z_K(\underline{EG},\underline{G})$ is the classifying space of the kernel of the projection 
\[\prod_{SK_1} G_i \longontoright \prod_{i=1}^n G_i.\]
Hence, if $G_1,\dots,G_n$ are finite groups, then Lemma \ref{lemma: no flag no K(pi,1)} shows that the kernel of this projection is torsion free, since $Z_K(\underline{EG},\underline{G})=K(\pi_1(Z_K(\underline{EG},\underline{G})),1)$ is of finite type by Lemma \ref{lemma: rel homotopy equiv}.
\end{rmk}

The following is an immediate corollary, which also proves Theorem \ref{thm: Z_K is an E-M space iff K flag INTRO}.

\begin{thm}\label{main.E-B.space}
Let $G_1,\dots , G_n$ be non-trivial discrete groups and $K$ be a simplicial complex with $n$ vertices. Then  $Z_K({\underline{BG}})$ is an Eilenberg-Mac Lane space if and only if $K$ is a flag complex. Equivalently, $Z_K(\underline{EG},\underline{G})$ is an Eilenberg-Mac Lane space if and only if $K$ is a flag complex.
\end{thm}
\begin{proof}
This follows from \cite{davis.okun}, Proposition \ref{lemma: no flag no K(pi,1)}, and the Denham-Suciu fibration.
\end{proof}

\section{Homotopy groups of other polyhedral products }\label{sec: htpy groups other Z_K}

\

Suppose $X$ is a path-connected, finite dimensional and pointed $CW$-complex. Then it is a classical result that $X$ is the classifying space of a topological group $G$, which can be described precisely, see for example \cite[Theorem 5.2]{milnor56b}. Hence, we can write $X\simeq BG$. This gives a homotopy equivalence $\Omega X \simeq G$, which implies an equivalence between the cone of the spaces $CG \simeq C(\Omega X)$. Let $\ast$ be the basepoint of $EG$. There is a commutative diagram of spaces
\begin{center}
\begin{tikzcd}
EG\times C(G) \arrow{r}{\simeq} & C(\Omega X) \\
\ast \times G \arrow{r}{\simeq} \arrow[hookrightarrow]{u} & \Omega X. \arrow[hookrightarrow]{u}
\end{tikzcd}
\end{center}
Hence, there is a homotopy equivalence of pairs $(EG,G)\simeq (C\Omega X, \Omega X)$. The Denham-Suciu fibration (\ref{D-S-fibration}) gives us 
\[Z_K(EG,G) \longrightarrow Z_K(BG,\ast) \longrightarrow BG^n,\]
which can also be written as 
\[Z_K(C\Omega X,\Omega X) \longrightarrow Z_K(X,\ast) \longrightarrow X^n,\]
because of the equivalence $(EG,G)\simeq (C\Omega X, \Omega X)$. This is an instance of the more general case of a sequence of $CW$-pairs $(\underline{X},\underline{\ast})$. Hence, there is a fibration
\[Z_K(\underline{ C\Omega X},\underline{\Omega X}) \longrightarrow Z_K(\underline{X},\underline{\ast}) \longrightarrow \prod_{i=1}^n X_i.\]
Assume that $G$ is path-connected. That means $\pi_0(G)=1$ and there are isomorphisms $\pi_1(X)=\pi_0(\Omega X)=\pi_0(G)=1$. That is, assume that $X$ is 1-connected. We want to prove that if $X$ is 1-connected, then $Z_K(EG,G) $ is 1-connected. 
%
\begin{thm}\label{thm: 1-connected polyhedral prod}
Let $X_1,\dots,X_n$ be 1-connected $CW$-complexes. Then the polyhedral product 
$Z_K(\underline{C \Omega X},\underline{\Omega X})$ is 1-connected.
\end{thm}
\begin{proof}
We prove this for $X_1=\cdots =X_n=X$, then the general proof is the same. As mentioned above, if $X$ is a 1-connected $CW$-complex, then $X\simeq BG$ and $G$ is path-connected, see \cite{milnor56b}. Moreover, there is an equivalence 
$$Z_K(C \Omega X,\Omega X) \simeq Z_K(EG,G),$$
where $EG$ and $G$ are path-connected. Hence, $Z_K(EG,G)$ is path-connected. In this proof we will work with $Z_K(EG,G)$.

To show that $\pi_1(Z_K(EG,G))=0$, the definition of the polyhedral product will be used. Recall that 
\[Z_K(EG,G)=\underset{{\sigma \in K}}{\colim} D(\sigma),\]
where $D(\sigma)$ is a product of $EG$'s and $G$'s. Also recall that $G\times \cdots \times G \subset D(\sigma)$, for all $\sigma \in K$, where $G\times \cdots \times G =D(\varnothing)$. For two simplices $\tau, \sigma \in K$ consider the pushout diagram
\begin{center}
\begin{tikzcd}
D(\tau) \cap D(\sigma) \arrow{r}{i} \arrow{d}[swap]{j} & D(\tau) \arrow{d} \\
D(\sigma) \arrow{r} & D(\tau) \cup_{D(\tau) \cap D(\sigma)} D(\sigma),
\end{tikzcd}
\end{center}
where $D(\tau) \cup_{D(\tau) \cap D(\sigma)} D(\sigma)$ is the colimit of the two maps emanating from the intersection $D(\tau) \cap D(\sigma)$. Using Seifert-van Kampen theorem for the fundamental group, it follows that 
\[\pi_1(D(\tau) \cup_{D(\tau) \cap D(\sigma)} D(\sigma))=\pi_1(D(\tau)) \ast_N \pi_1(D(\sigma)),\]
where $N$ is the subgroup generated by the images of the fundamental group of the intersection under the induced maps of $i$ and $j$ in $\pi_1$. Let the set $V_{\sigma, \tau}=\{v_1,\dots,v_t\}$ be the maximal set of vertices in $K$ such that $V_{\sigma, \tau}\cap \sigma = \varnothing$ and $V_{\sigma, \tau}\cap \tau = \varnothing$.  Clearly, $\pi_1(D(\tau)) \ast_N \pi_1(D(\sigma))=\pi_1(G^t)$, since $i$ and $j$ induce monomorphisms in $\pi_1$ and the images of these two maps do not hit the coordinates corresponding to the vertices in $V_{\sigma, \tau}$.

$K$ contains only a finite number of simplices $\tau_1,\dots,\tau_k$. Let $V_{i_1,\dots,i_l}$ be the maximal set of vertices in $K$ such that $V_{i_1,\dots,i_l} \cap \tau_{i_j}=\varnothing$ for $1 \leq j \leq l \leq k$. As explained above
\[\pi_1(D(\tau_1) \cup_{D(\tau_1)\cap D(\tau_2)} D(\tau_2))=\pi_1(D(\tau_1)) \ast_{N_1} \pi_1(D(\tau_2))=\pi_1(G^{|V_{1,2}|}).\]
To complete the proof, first perform the computation by taking the colimit with more simplices until all simplices $\tau_1,..,\tau_{k-1}$ are used. It follows that
\[\pi_1(\underset{{1 \leq i \leq k-1}}{\colim} D(\tau_i))=\pi_1(G^{|V_{1,\dots,k-1}|}).  \]
In the last step, the polyhedral product equals
\[Z_K(EG,G) = \underset{{1 \leq i \leq k-1}}{\colim} D(\tau_i) \cup_{\underset{{1 \leq i \leq k-1}}{\colim} D(\tau_i) \cap D(\tau_k)} D(\tau_k).  \]
Thus, the fundamental group equals
\[\pi_1(\underset{{1 \leq i \leq k-1}}{\colim} D(\tau_i) \cup_{\underset{{1 \leq i \leq k-1}}{\colim} D(\tau_i) \cap D(\tau_k)} D(\tau_k))=\pi_1(\underset{{1 \leq i \leq k-1}}{\colim} D(\tau_i)) \ast_{N_{k-1}} \pi_1(D(\tau_k)). \]
Since $V_{1,\dots,k}=\varnothing$, it follows that $\pi_1(Z_K(EG,G))=0$.
\end{proof}

Now using Hurewicz's theorem, we can describe the first non-vanishing homotopy group of 
$Z_K(\underline{C \Omega X},\underline{\Omega X})$.

\begin{prop}\label{prop: 2nd homotopy group, special case}
Let $X_1,\dots,X_n$ be 1-connected $CW$-complexes and $K$ be a simplicial complex with $n$ vertices. If $Z_K(\underline{C \Omega X},\underline{\Omega X})$ is $(k-1)$-connected, then the first non-vanishing homotopy group is given by 
$$\pi_k( Z_K(\underline{C \Omega X},\underline{\Omega X})) \cong \bigoplus_{I \subset \{1,\dots,n\}}H_k(|K_I| \wedge \widehat{\underline{\Omega X}}^I; \Z).$$
\end{prop}
\begin{proof}
From Theorem \ref{thm: 1-connected polyhedral prod} we have that $Z_K(\underline{C \Omega X},\underline{\Omega X})$ is 1-connected. Furthermore assume that the first non-vanishing homotopy group is the $k-$th homotopy group. Then from the Hurewicz theorem we have 
$$
\pi_k(Z_K(\underline{C \Omega X},\underline{\Omega X})) \cong H_k(Z_K(\underline{C \Omega X},\underline{\Omega X}),\Z).
$$
It follows from \cite[Theorem 2.10]{cohen.macs} that there is a homotopy equivalence
$$
\Sigma Z_K(\underline{C \Omega X},\underline{\Omega X}) \to \Sigma\left( \bigvee_{I\subset [n]} \widehat{Z}_{K_I}(\underline{C \Omega X},\underline{\Omega X})  \right).
$$
Since $C \Omega X_i$ is contractible for all $i$, it follows from \cite[Theorem 2.19]{cohen.macs} that there is a homotopy equivalence
$$
\widehat{Z}_{K_I}(\underline{C \Omega X},\underline{\Omega X}) \to |K_I|\wedge \widehat{\underline{\Omega X}}^{I}.
$$
Hence, it follows that there is a homotopy equivalence 
$$
\Sigma Z_K(\underline{C \Omega X},\underline{\Omega X}) \to \Sigma\left( \bigvee_{I\subset [n]} |K_I|\wedge \widehat{\underline{\Omega X}}^{I}  \right).
$$
Finally, we have isomorphisms 
$$
H_k( Z_K(\underline{C \Omega X},\underline{\Omega X})) \cong H_k\left( \bigvee_{I\subset [n]} |K_I|\wedge \widehat{\underline{\Omega X}}^{I};\Z  \right)\cong \bigoplus_{I\subset [n]}H_k(|K_I|\wedge \widehat{\underline{\Omega X}}^{I};\Z ).
$$
\end{proof}

Another question is the case when a topological group $G$ acts freely and properly discontinuously on a $CW$-complex $Y$ and $p:Y \to X=Y/G$ is a bundle projection. There is a lemma due to Denham and Suciu \cite{denham} which describes the fibre when comparing certain fibrations involving polyhedral products.
\begin{lemma}[Denham \& Suciu]
Let $p: (E,E') \to (B,B')$ be a map of pairs, such that both $p:E \to B$ and $p'=p|_{E'}: E' \to B'$ are fibrations with fibres $F$ and $F'$ respectively. Suppose that either $F=F'$ or $B=B'$. Then the product fibration $p^{\times n}: E^n \to B^n$ restricts to a fibration 
\begin{equation}\label{eqn:D.S.Lemma}
Z_K(F,F') \longrightarrow Z_K(E,E') \xrightarrow{Z_K(p)} Z_K(B,B').
\end{equation}
Moreover, if $ (F,F') \to (E,E') \to (B,B')$ is a relative bundle (with structure group $G$), and either $F=F'$ or $B=B'$, then (\ref{eqn:D.S.Lemma}) is also a bundle (with structure group $G^n    $).
\end{lemma}

Consider the relative map $p: (Y,G) \to (Y/G,\ast)$ with fibres $F=F'=G$. Thus, there is a fibration 
\[G^n \longrightarrow Z_K(Y,G) \longrightarrow Z_K(Y/G,\ast).  \]
Pushing the fibre to the right by taking its classifying space, one gets a fibration
\[Z_K(Y,G) \longrightarrow Z_K(Y/G,\ast) \longrightarrow BG^m.  \]
In this case it is not true in general that $Z_K(Y,G)$ is 1-connected, as shown in the next example.

\begin{ex}
Let $G=\Z/2\Z$ act on the 2-sphere $S^2$ by the antipodal map. Then the orbit space is $S^2/(\Z/2\Z)=\R P^2$. Hence there is a fibration
\[Z_K(S^2,\Z/2\Z) \longrightarrow Z_K(\R P^2,\ast) \longrightarrow B(\Z/2\Z)^m. \]
If $K$ is the 0-skeleton of the 1-simplex, then 
$$Z_K(S^2,\Z/2\Z)=(S^2 \times \Z/2\Z) \cup (\Z/2\Z \times S^2) $$ 
is equivalent to the wedge sum $(\bigvee_{4}S^2)\vee S^1$. Hence, $Z_K(S^2,\Z/2\Z)$ is not simply connected.
\end{ex}

Nevertheless, the following theorem shows that the fundamental group of these polyhedral products can be computed with the information that was obtained from Section \ref{sec: htpy groups Z_K for BG}.

\begin{lemma}\label{lemma: iso fundamental groups of pairs}
If $A \subset X \subset Y $ with $A$ finite discrete, $X$ and $Y$ path-connected, such that the induced map $\pi_i(X) \xrightarrow{i_{\#}} \pi_i(Y)$ is an isomorphism for $i=0,1$, then 
\[\pi_1(Z_K(X,A)) \cong \pi_1(Z_K(Y,A)). \]
\end{lemma}
\begin{proof}
Let $D_X(\sigma)$ denote the functor $D$ evaluated on the pair $(X,A)$ and $D_Y(\sigma)$ denote the functor $D$ evaluated on the pair $(Y,A)$. If $\sigma \in K$ is a simplex and if $\pi_i(X) \to \pi_i(Y)$ is an isomorphism for $i=0,1$, then the map
$D_X(\sigma) \hookrightarrow D_Y(\sigma)$
induces an isomorphism on the fundamental group. By definition,
\[\pi_1(Z_K(X,A))=\pi_1(\underset{{\sigma \in K}}{\colim} D_X(\sigma))\] 
and
\[\pi_1(Z_K(Y,A))=\pi_1(\underset{{\sigma \in K}}{\colim} D_Y(\sigma)). \] 
Recall that if $X_2$ is the 2-skeleton of $X$, then $\pi_1(X) \cong \pi_1(X_2)$. Similarly, if $Y_2$ is the 2-skeleton of $Y$, then $\pi_1(Y) \cong \pi_1(Y_2)$. So this shows that $\pi_1(X_2) \cong \pi_1(Y_2)$. Similarly, it suffices to work with the 2-skeleton of the spaces $Z_K(X,A)$ and $ Z_K(Y,A)$. Denote these two spaces by $Z_K(X,A)_2$ and $ Z_K(Y,A)_2$, respectively. It is clear that there are inclusions up to homotopy $Z_K(X,A)_2 \hookrightarrow Z_{SK_1}(X_2,A)$ and $Z_K(Y,A)_2 \hookrightarrow Z_{SK_1}(Y_2,A)$ since adding a simplex $\sigma$ on $k\geq 3$ vertices, does not change the 2-skeleton of the polyhedral product. Both inclusions induce isomorphisms in $\pi_1$. This suffices.
\end{proof}

Note that Lemma \ref{lemma: iso fundamental groups of pairs} can be proved using the fact that there is a homotopy equivalence
\[Z_K(X,A))=\underset{{\sigma \in K}}{\colim}D_X(\sigma) \simeq \underset{{\sigma \in K}}{\hocolim}\hspace{.15cm} D_X(\sigma),\]
see \cite{cohen.macs}, and the fact that 
\[\pi_1 ( \underset{{\sigma \in K}}{\hocolim} D_X(\sigma)) \cong \underset{{\sigma \in K}}{\colim} \pi_1( D_X(\sigma)),\]
see \cite{farjoun} for details.

\

Next consider the pair $(BG,BH)$, where $H$ is a closed subgroup of $G$. As mentioned in the introduction, there is an inclusion $G/H \hookrightarrow C(G/H)$ which gives an inclusion $EG\times G/H \hookrightarrow EG \times C (G/H)$. Hence, $G/H$ can be regarded as a $G$-equivariant subspace of $EG$. There is a fibration
\[ Z_K(EG,G/H) \longrightarrow (EG)^n \times_{G^n}Z_K(EG,G/H) \longrightarrow (BG)^n .\]
Since  $EG\times G/H \simeq G/H$ and $EG\times_G G/H \cong EG/H \cong BH$, then the space $(EG)^n \times_{G^n}Z_K(EG,G/H)$ is homotopy equivalent to 
$$Z_K(EG/G,EG\times_G G/H) \simeq Z_K(BG,BH) .$$ 
This proves the following proposition 
\begin{prop}
Let $H$ be a closed subgroup of the Lie group $G$. There is a fibration given by
\[ Z_K(EG,G/H) \longrightarrow Z_K(BG,BH)  \longrightarrow (BG)^n .\]
\end{prop}
\begin{prop}
Let $H$ be a closed subgroup of $G$. Then there is a homotopy equivalence
\[\Omega ( Z_K(BG,BH) ) \simeq G^n \times \Omega (Z_K(EG,G/H))\] 
and a short exact sequence of groups
\[1 \to\pi_1(Z_K(EG,G/H)) \to  \pi_1( Z_K(BG,BH) ) \to G^n \to 1.\]
This decomposition may not preserve the group structure.
\end{prop}
\begin{proof}
There is a fibration as follows
\[ Z_K(EG,G/H) \longrightarrow Z_K(BG,BH) \longrightarrow (BG)^n .\]
Taking the loops of this fibration gives a fibration
\[\Omega (Z_K(EG,G/H)) \longrightarrow \Omega (Z_K(BG,BH)) \longrightarrow \Omega((BG)^n)=G^n .\]
There is a section $G^n \xrightarrow{s} \Omega ( Z_K(BG,BH) )$ which implies that there is a homotopy equivalence of topological spaces 
\[\Omega ( Z_K(BG,BH) ) \simeq G^n \times \Omega (Z_K(EG,G/H)).\]
The short exact sequence follows from this equivalence.
\end{proof}

\section{An extension problem}\label{section.extension}

In this section we are interested in investigating a certain extension problem that arises from polyhedral products. 

A. Adem, F. Cohen and E. Torres Giese \cite{fredb2g} introduced the spaces $B(q,G)$ for integers $q \geq 2$, which form a sequence of spaces and a filtration of $BG$
\[
 B(2,G) \subset B(3,G) \subset \cdots \subset B(\infty,G)=BG,
\]
In particular, $B(2,G)$ is defined to be the geometric realization
\[B(2,G)=\big(\bigsqcup_{k\geq 1}\Ho(\Z^k,G)\times \Delta[k]\big)/\sim ,\]
where $\sim$ is generated by the standard face and degeneracy relations, see \cite{milnor57}.

\

A proof of the following lemma can be found in \cite{fredb2g}. 

\begin{lemma}
Let $G$ be a nonabelian group. The following are equivalent
\begin{itemize}\label{lemma:TC.equivalent.definitions}
\item[a.] If $g \notin Z(G)$, then $C(g)$ is abelian.
\item[b.] If $[g,h]=1$, then $C(g)=C(h)$ whenever $g,h \notin Z(G)$.
\item[c.] If $[g,h]=1=[h,k]$, then $[g,k]=1$ whenever $h \notin Z(G)$.
\item[d.] If $A,B \leq G$ and $Z(G) <C_G(A)\leq C_G(B) < G$, then $C_G(A) = C_G(B)$.
\end{itemize}
\end{lemma}
\begin{defn}\label{defn: TC group}
Let $G$ be a nonabelian group. If $G$ satisfies any of the equivalent statements in Lemma \ref{lemma:TC.equivalent.definitions}, then $G$ is called a \textit{transitively commutative} group, or simply a \textit{TC group}. 
\end{defn}
Some examples of TC groups include abelian groups and dihedral groups.
Let $a_1,\dots,a_k \in G - Z(G) $ be representatives of distinct centralizers in $G$ such that the set $\{C_G(a_1),\dots,C_G(a_k)\}$ is the smallest collection of centralizers covering $G$. In \cite{fredb2g}  A. Adem, F. Cohen and E. Torres Giese prove the following theorem.
\begin{thm}[Adem, Cohen \& Torres Giese] If $G$ is a finite TC group with trivial center, then there is a homotopy equivalence
\[B(2,G)\simeq \bigvee_{1 \leq i \leq k}\Big( \prod_{p||C_G{a_i}|} BP\Big), \] 
where $P\in {\rm Syl}_P(G)$.
\end{thm}

The theorem states that if $G$ is a TC group with trivial center, then 
\[ B(2,G) \simeq BG_1 \vee \cdots \vee BG_k, \]
where $G_i=\prod_{p||C_G{a_i}|} BP$, and this gives the polyhedral product $Z_{K_0}(\underline{BG})$. Hence, there is a map
\[Z_{K_0}(\underline{BG}) \longrightarrow BG. \]
On the level of fundamental groups, this gives the homomorphism
\[
G_1 \ast \dots \ast G_k \overset{\varphi}{\longrightarrow} G.
\]

In general, assume $\{G_1,\dots,G_n\}$ is a family of subgroups of a finite discrete group $G$. Then there is a natural map 
\[
G_1 \ast \dots \ast G_n \overset{\varphi}{\longrightarrow} G.
\]
It is natural to ask the following: for what abstract simplicial complexes $K$ on $[n]$ does the map
$
BG_1 \vee \dots \vee BG_n \xrightarrow{B\varphi} BG
$
extend to $Z_K(\underline{BG})$? That is equivalent to asking for those $K$ such that the following diagram commutes
\begin{center}
\begin{tikzcd}
BG_1 \vee \dots \vee BG_n \arrow{r}{B\varphi}\arrow[hookrightarrow]{d}{i} &BG \\
Z_K(\underline{BG}). \arrow{ur} &
\end{tikzcd}
\end{center}
The importance of this question is that, if the map in question extends, then we detect commuting elements in $G$. One could also pose the same question algebraically and seek the simplicial complexes $K$ for which the following diagram commutes
\begin{center}
\begin{tikzcd}
G_1 \ast \cdots \ast G_n \arrow{r}{\varphi} \arrow{d}{i_{\#}} &G \\
\prod_{SK_1}G_i .\arrow{ur}
\end{tikzcd}
\end{center}
The rest of this section is devoted to understanding this question.

\

The first example is that of a transitively commutative (TC) group $G$ with trivial center, where the family of subgroups $\{G_1,\dots,G_n\}$ is chosen to consist of the maximal abelian subgroups of $G$.
\begin{prop}
Let $G$ be a transitively commutative group and let the set $\{A_1,\dots,A_k\}$ be the distinct maximal abelian subgroups of $G$. Then the map 
$$B\varphi : BA_1 \vee \dots \vee BA_k \to BG$$ 
does not extend for any simplicial complex $K$ such that $K_0 \neq K$.
\end{prop}
\begin{proof}
If it extends for a simplicial complex $K$, then $K$ has at least one edge, say $\{i,j\}$. This implies that $[A_i,A_j]=1$ for some $i \neq j$. One can check that the groups $A_i$ are centralizers of elements in $G$ not in the center, and that they intersect trivially, yielding a contradiction to the assumption.
\end{proof}

Recall that for a group $G$ there is a normal series 
called the \textit{descending central series} of $G$ given by
\[G=\Gamma^1(G) \unrhd \Gamma^2(G) \unrhd \cdots \unrhd \Gamma^n(G) \unrhd \cdots, \]
where $\Gamma^2(G)=[G,G]$ and $\Gamma^{n+1}(G)=[\Gamma^{n}(G),G]$ for all $n\geq 2$, and $\Gamma^k(G)$ is called the $k$-th stage series. 

\

Next we extend the notion of TC-groups by defining a similar class of groups with similiar properties, which eventually include TC-groups as a special case.

\begin{defn}
Let $g_l,k_l \in \Gamma^{l-1}(G)$. We say that $G$ is an \textit{$l$-transitively commutative group}, or simply $l$-TC group, if 
\[
[g_l,h]=1=[h,k_l] \Longrightarrow [g_l,k_l]=1 \text{ for all }h \in G.
\]
\end{defn}
Note that a $2$-TC group is the same as the ordinary TC group. Also that the condition that $l\geq 2$ is required for the definition.  
\begin{rmk}
If $G$ is an $m$-TC group then it is an $n$-TC group for all $n\geq m$. This follows from the structure of the descending central series. Hence, if $G$ is a TC group, then it is a $k$-TC group for all $k\geq2$. By convention, let finite simple groups be called \textit{$1$-TC groups}.
\end{rmk}
\begin{ex}
A group which is 3-TC but not 2-TC is the quaternion group, $Q_8$. This group has a presentation as follows
\[
Q_8=\{\pm1,\pm i, \pm j, \pm k | i^2=j^2=k^2=-1, ij=k,jk=i,ki=j\}.
\]
First we note that $[i,-1]=1=[-1,j]$ but $[i,j]=-1$, so $Q_8$ is not 2-TC. Now $[i,j]=[j,k]=[k,i]=-1$ and $[[i,j],h]=1$ for any $h\in Q_8$. Therefore $Q_8$ is 3-TC. One can also compute the descending central series for $Q_8$. Its commutator subgroup is $\Gamma^2(Q_8)[Q_8,Q_8]=\{\pm1\} \cong \mathbb{Z}/2\mathbb{Z} $ and $\Gamma^3(Q_8)=[[Q_8,Q_8],Q_8]=1$. Therefore, the descending central series for $Q_8$ is
\[
1 \lhd \mathbb{Z}/2\mathbb{Z} \lhd Q_8.
\]
\end{ex}
\begin{lemma}
If $G$ is nilpotent of nilpotency class $m$, then $G$ is an $m$-TC group.
\end{lemma}
\begin{proof}
$G$ is nilpotent of nilpotency class $m$ means $\Gamma^m(G)=1$, that is, $[g_m,h]=1$ for all $g_m \in \Gamma^{m-1}(G), h\in G$. In particular, $[g_m,k_m]=1$ for all elements $g_m,k_m \in \Gamma^{m-1}(G)$, hence $G$ is $m$-TC.
\end{proof}

\begin{defn}
Let $g\in G$. The \textit{$l$-stage centralizer of $g\in G$}, denoted $C^l_G(g)$, is the subgroup
\[ C^l_G(g):=\{g_{l+1}\in \Gamma^l(G)| g_{l+1} g g_{l+1}^{-1} = g \}. \] 
We can also write $C^l_G(g)= C_G(g)\cap \Gamma^l(G)$.
\end{defn}

\begin{lemma}
Let $G$ be an $(l+1)$-TC group with trivial center. Then $C^l_G(g)$ are abelian subgroups of $G$ for all $g\in G$. Moreover, if $Z(G)=1$, then $C^l_G(g)$ are distinct and either intersect trivially or coincide.
\end{lemma}
\begin{proof}
It is clear that $C^l_G(g)$ are abelian. If $C^l_G(g) \cap C^l_G(h)$ is non-trivial then there is $1\neq k \in C^l_G(g) \cap C^l_G(h)$ and thus, the two coincide by definition of an $(l+1)$-TC group.
\end{proof}

\begin{lemma}\label{lemma:k-stage centralizer}
Let $\{C^l_G(g)\}_{g \in \Lambda}$ be the family of distinct $l$-stage centralizers of elements in $G$. If $G$ is an $(l+1)$-TC group with $Z(G)=1$, the map 
\[
\bigvee_{g \in \Lambda} BC^l_G(g) \longrightarrow BG
\]
does not extend for any $K$ on $[m]$ where $m=|\{C^l_G(g)\}_{g \in \Lambda}|$.
\end{lemma}
\begin{proof}
This follows from the previous lemma.
\end{proof}

\begin{cor}\label{corollary: criterion for non extension, k-TC gps}
Let $G$ be a finite $(k+1)$-TC group with trivial center. Let $G_1,G_2$ be two subgroups of $G$ such that $C^k_G(g_1)\leq G_1$ and $C^k_G(g_2)\leq G_2$, where $C^k_G(g_1)\cap C^k_G(g_2)= 1$. Then the map 
\[
BG_1 \vee BG_2 \longrightarrow BG
\]
does not extend to $BG_1 \times BG_2$.
\end{cor}
\begin{proof}
The proof follows from Lemma \ref{lemma:k-stage centralizer}.
\end{proof}

Using this corollary the following theorem follows easily. This theorem applies to TC groups in particular.
\begin{thm}\label{thm: non extension criterion}
Let $\{G_i\}_{i=1}^m$ be a family of subgroups of a finite $(k+1)$-TC group $G$ with trivial center. If $G_i$ contain distinct $C^l_G(g_i)$, for $i=1,2,\dots,m$ and $l\geq k$, then the map 
\[
\bigvee_{1 \leq i \leq m} {BG_i} \longrightarrow BG
\]
does not extend to $Z_K(\underline{BG})$ for any $K$ on $[m]$, other than the 0-skeleton.
\end{thm}

\begin{proof}
From the definition we have $C^l_G (g_i)=C_G(g_i)\cap \Gamma^l(G)$ which implies $C^l_G (g_i) \geq C^{l+1}_G (g_i)$. The rest of the proof follows from Corollary \ref{corollary: criterion for non extension, k-TC gps}.
\end{proof}

Below we give a result for when this extension actually exists.
\begin{defn}
Let $K$ be a simplicial complex. The \textit{flag complex of $K$} is  the simplicial complex $\Flag(K)$ obtained from $K$ by completing the minimal non-faces to faces.
\end{defn}
Note that an edge $\{i,j\}$ is in $K$ if and only if it is an edge in $\Flag(K)$. Now let $\{H_1,\dots,H_n\}$ be a collection of subgroups of a group $G$. Let $\Gamma$ be a graph on $n$ vertices that records the commutativity relations of the set $\{H_1,\dots,H_n\}$. That means, $\{i,j\}$ is an edge in $\Gamma$ if and only if $[H_i,H_j]=1$. Let $\Flag(\Gamma)$ be flag complex of $\Gamma$.
\begin{lemma}\label{lemma:extension to Z_flag(L)}
Let $\{H_1,\dots,H_n\}$ be a collection of subgroups of a finite group $G$ with trivial center. If $\Gamma$ is the graph described above, then the map
\[
BH_1 \vee \cdots \vee BH_n \longrightarrow BG
\]
extends to $Z_{\Flag(\Gamma)}(\underline{BH})$.
\end{lemma}
\begin{proof}
By definition, ${\Flag(\Gamma)}$ is a flag complex. Hence, the polyhedral product $Z_{\Flag(\Gamma)}(\underline{BH})$ is an Eilenberg-Mac Lane space with fundamental group $ \prod_{\Gamma}H_i$. So there is a commutative diagram of groups
\begin{center}
\begin{tikzcd}
H_1 \ast \cdots \ast H_n \arrow{r}{\varphi} \arrow{d}[swap]{i_{\#}}  & G \arrow{d}{id}\\
{\prod_{\Gamma}H_i }\arrow{r}{\widetilde{\varphi}} & G.
\end{tikzcd}
\end{center}
Hence, there is a commutative diagram of spaces
\begin{center}
\begin{tikzcd}
BH_1 \vee \cdots \vee BH_n \arrow{r}{B\varphi} \arrow{d}[swap]{Bi_{\#}}  & BG \arrow{d}{id}\\
Z_{\Flag(\Gamma)}(\underline{BH}) \arrow{r}{B\widetilde{\varphi}} & BG.
\end{tikzcd}
\end{center}
\end{proof}

This lemma actually holds for any simplicial complex $K$ with the property that $\Gamma \subseteq K \subseteq \Flag(\Gamma)$, even if $K$ is not a flag complex. 
\begin{rmk}
If the collection of groups $\{H_1,\dots,H_n\}$ in Lemma \ref{lemma:extension to Z_flag(L)} is taken to be the collection of groups $\{G_1,\dots,G_n\}$ in Theorem \ref{thm: non extension criterion}, then the graph $\Gamma$ is a discrete set with $n$ points, thus no non-trivial extension is obtained.
\end{rmk}
\begin{thm}
Let $\{H_1,\dots,H_n\}$ be a collection of subgroups of a finite group $G$ with $Z(G)=1$. If $\Gamma$ is the graph described above and $K$ is a simplicial complex such that $\Gamma \subseteq K \subseteq \Flag(\Gamma)$, then the map 
\[
BH_1 \vee \cdots \vee BH_n \longrightarrow BG
\]
extends to $Z_{K}(\underline{BH})$.
\end{thm}
\begin{proof}
The proof follows by naturality and Lemma \ref{lemma:extension to Z_flag(L)}. First, the following diagram of inclusions commutes
\begin{center}
\begin{tikzcd}
BH_1 \vee \cdots \vee BH_n \arrow[hookrightarrow]{rr}{Bi_{\#}=i} \arrow[hookrightarrow]{dr}[swap]{i_1}  &     & Z_{\Flag(\Gamma)}(\underline{BH}) \\
     & Z_{K}(\underline{BH}) \arrow[hookrightarrow]{ur}[swap]{i_2} &    .   
\end{tikzcd}
\end{center}
Combining this diagram with the diagram of spaces in Lemma \ref{lemma:extension to Z_flag(L)}, it follows that the diagram
\begin{center}
\begin{tikzcd}
BH_1 \vee \cdots \vee BH_n \arrow{r}{=} \arrow[hookrightarrow]{d}[swap]{i_1} &  BH_1 \vee \cdots \vee BH_n \arrow{r}{B\varphi} \arrow{d}[swap]{Bi_{\#}}  & BG \arrow{d}[swap]{id}\\
Z_{K}(\underline{BH}) \arrow[hookrightarrow]{r}{i_2} & Z_{\Flag(\Gamma)}(\underline{BH}) \arrow{r}{B\widetilde{\varphi}} & BG
\end{tikzcd}
\end{center}
commutes. Hence, the map $BH_1 \vee \cdots \vee HB_n \xrightarrow{B\varphi} BG$ extends to $Z_{K}(\underline{BH})$. That is, the following diagram of spaces commutes
\begin{center}
\begin{tikzcd}
BH_1 \vee \dots \vee BH_n \arrow{r}{B\varphi}\arrow[hookrightarrow]{d}[swap]{i_1} &BG \\
Z_K(\underline{BH}) \arrow{ur}[swap]{B\widetilde{\varphi}\circ i_2} & .
\end{tikzcd}
\end{center}
\end{proof}

\section*{Acknowledgments}
The author would like to thank Fred Cohen for his suggestions and insightful comments, Ali Al-Raisi for the discussions on the subject, and Kouyemon Iriye for some comments on the paper.

The author is supported by DARPA grant number N66001-11-1-4132.

\end{document}